\documentclass[10pt]{article}
\usepackage{graphicx}
\usepackage{tikz}
\usepackage[margin=1in]{geometry}
\usepackage{amsmath, amssymb, amsfonts, amsthm}
\newtheorem{lem}{\noindent {\bf Lemma}}[section]
\newtheorem{prop}{\noindent {\bf Proposition}}[section]
\newtheorem{coro}{\noindent {\bf Corollary}}[section]
\newtheorem{thm}{\noindent {\bf Theorem}}[section]

\newcounter{remark}
\newenvironment{remark}{\smallskip\noindent {\bf Remark \arabic{section}.\arabic{remark}.}}
{\addtocounter{remark}{1}\par}
\newcounter{example}
\newenvironment{exa}{\smallskip\noindent {\bf Example \arabic{section}.\arabic{example}.}}
{\addtocounter{example}{1}\par}
\catcode`@=11 
\date{}
\newcounter{defi}
\newenvironment{defi}{
\smallskip \noindent
{\bf
  Definition \arabic{section}.\arabic{defi}.
}}{\addtocounter{defi}{1}\par}
\newcommand{\ZZ}{\mathbb Z}
\newcommand{\NN}{\mathbb N}
\newcommand{\RR}{\mathbb R}

\title{\bf Examples of metric spaces with asymptotic property~$C$}
\author{\large  Yan Wu$^\ast$\qquad Jingming Zhu$^{\ast\ast}$
\footnote{
College of Mathematics Physics and Information Engineering, Jiaxing University, Jiaxing , 314001, P.R.China.
$^\ast$ E-mail: yanwu@mail.zjxu.edu.cn $^\ast\ast$ E-mail: jingmingzhu@mail.zjxu.edu.cn }}
\date{}

\begin{document}
\maketitle
\begin{center}
\begin{minipage}{0.9\textwidth}
\noindent{\bf Abstract.}
We construct a class of metric spaces whose transfinite asymptotic dimension and complementary-finite asymptotic dimension are both $\omega+k$ for any $k\in\mathbb{N}$, where $\omega$ is the smallest infinite ordinal number and a metric space whose transfinite asymptotic dimension and complementary-finite asymptotic dimension are both $2\omega$.
Moreover, we study the relationship between asymptotic dimension growth, transfinite asymptotic dimension and finite decomposition complexity.
\end{minipage}

\end{center}
\footnote{
This research was supported by
the National Natural Science Foundation of China under Grant (No.11871342,11801219)

}

{\bf Acknowledgments.} The author wish to thank the reviewers for careful
reading and valuable comments. This work was supported by NSFC grant of P.R. China (No.11871342,11801219). And the authors want to thank V.M. Manuilov, Jiawen Zhang and Benyin Fu for helpful discussion.

\begin{section}{Introduction}\

Asymptotic dimension introduced by M.Gromov \cite{Gromov} is a fundamental concepts in coarse geometry.
It starts to draw much attention after Yu proved Novikov higher signature conjecture holds for groups with finite asymptotic dimension(FAD). As an excellent survey, G. Bell and A. Dranishnikov summarized the definitions, properties and examples of FAD in \cite{Bell2011}. However, there are many groups and spaces with infinite asymptotic dimension, which contains wreath product $\mathbb{Z}\wr\mathbb{Z}$, Thompson groups, etc. Generalizing FAD, A. Dranishnikov \cite{polynomialdimensiongrowth} defined the asymptotic dimension growth for a metric space; if the asymptotic dimension growth function is eventually constant, then the metric space has FAD. He showed that the wreath product of a finitely generated nilpotent group with a finitely generated FAD group has polynomial asymptotic dimension growth. He also showed that polynomial asymptotic dimension growth implies Yu's Property A and hence the coarse Baum-Connes conjecture, provided the space has bounded geometry (see \cite{Yu}).

As a large scale analogue of Haver's property C \cite{harveC} in dimension theory, A. Dranishnikov introduced the notion of asymptotic property C(APC) in \cite{PropertyC}, which covers a large family of metric spaces with infinite asymptotic dimension. To classifying the metric spaces with APC, T. Radul defined the transfinite asymptotic dimension (trasdim) which can be viewed as a transfinite extension of asymptotic dimension and proved that for a metric space $X$, trasdim$(X)<\infty$ if and only if $X$ has APC in \cite{Radul2010}.  He also gave examples of metric spaces with trasdim$=\infty$ and with trasdim$=\omega$, where $\omega$ is the smallest infinite ordinal number (see \cite{Radul2010}). But whether there is a metric space $X$ with $\omega<$trasdim$(X)<\infty$ (stated as"omega conjecture"in \cite{satkiewicz} by M. Satkiewicz) is unknown until recently (see \cite{omega+1}).

Inspired by the technique in \cite{omega+1},
we construct a class of metric spaces $X_{\omega+k}$ with
\[
\text{trasdim}(X_{\omega+k})=\text{coasdim}(X_{\omega+k})=\omega+k \text{ for any } k\in \mathbb{N},
\]
which generalized the result in \cite{omega+1}. And  we use these metric spaces to construct a metric space $Y_{2\omega}$ with trasdim$(Y_{2\omega})=$coasdim$(Y_{2\omega})=2\omega$. \

Both asymptotic dimension growth and transfinite asymptotic dimension are concepts generalizing asymptotic dimension. However, for any unbounded increasing function $g$, we construct a metric space $X_{\omega}(g)$ such that
\[
\text{trasdim$(X_{\omega}(g))=\omega$, $X_{\omega}(g)\in\mathcal{D}_{\omega}$ and $ad_{X_{\omega}(g)}(r)\geq g(r)$ for every $r>0$.}
\]
After estimating the asymptotic dimension growth of $Y_{2\omega}$,
we find a metric space $X_{\omega}(g)$ such that
\[
\text{trasdim$(Y_{2\omega})=$2trasdim$(X_{\omega}(g))$ and $ad_{X_{\omega}(g)}(r)> ad_{Y_{2\omega}}(r)$ for every $r>0$,}
\]
which shows that transfinite asymptotic dimension and the exact finite decomposition complexity are independent of asymptotic dimension growth.

The paper is organized as follows. In Section 2, we give some preliminaries
on asymptotic dimension and complementary-finite asymptotic dimension. Section 3 and Section 4 deal with the construction of metric spaces with infinite asymptotic dimension: we introduce a concrete metric space $X_{\omega+k}$ and prove that transfinite asymptotic dimension and complementary-finite asymptotic dimension of $X_{\omega+k}$ are both $\omega+k$ which are essential for the construction of the metric space $Y_{2\omega}$ with trasdim$(Y_{2\omega})=$coasdim$(Y_{2\omega})=2\omega$. In Section 5,
we show that asymptotic dimension growth has nothing to do with transfinite asymptotic dimension by examples and similar result holds for finite decomposition complexity and asymptotic dimension growth.
\end{section}

\begin{section}{Preliminaries}\

The notion of asymptotic dimension was first introduced by Gromov in 1992 \cite{Gromov} as a coarse analogue of the classical Lebesgue topological covering dimension.
Our terminology concerning the asymptotic dimension follows from \cite{Bell2011} and for undefined terminology we refer to \cite{Radul2010}. \

Let~$(X, d)$ be a metric space and $U,V\subseteq X$, let
\[
\text{diam}~ U=\text{sup}\{d(x,y)| x,y\in U\}
\text{   and   }
d(U,V)=\text{inf}\{d(x,y)| x\in U,y\in V\}.
\]
Let $R>0$ and $\mathcal{U}$ be a family of subsets of $X$. $\mathcal{U}$ is said to be \emph{$R$-bounded} if
\[
\text{diam}~\mathcal{U}\doteq\text{sup}\{\text{diam}~ U~|~ U\in \mathcal{U}\}\leq R.
\]
In this case, $\mathcal{U}$ is said to be \emph{uniformly bounded}.
For $r>0$, a family $\mathcal{U}$ is said to be\emph{ $r$-disjoint} if
\[
d(U,V)> r~\text{for every}~ U,V\in \mathcal{U}\text{~with~}U\neq V.
\]

In this paper, we denote
$\bigcup\{U~|~U\in\mathcal{U}\}$ by $\bigcup\mathcal{U}$ and denote $\{U~|~U\in\mathcal{U}_{1}\text{~or~}~U\in\mathcal{U}_{2}\}$
by $\mathcal{U}_{1}\cup\mathcal{U}_{2}$.
Let $A$ be a subset of $X$, we denote $\{x\in X|d(x,A)<\epsilon\}$ by $N_{\epsilon}(A)$, denote $\{x\in X|d(x,A)\leq\epsilon\}$ by $\overline{N_{\epsilon}(A)}$ for some $\epsilon>0$.
We also denote $\{N_{\delta}(U)~|~U\in\mathcal{U}\}$ by $N_{\delta}(\mathcal{U})$ for some $\delta>0$.\

\begin{defi}
A metric space $X$ is said to have \emph{finite asymptotic dimension} if
there exists $n\in\NN$, such that for every $r>0$,
there exists a sequence of uniformly bounded families
$\{\mathcal{U}_{i}\}_{i=0}^{n}$ of subsets of $X$
such that the family
$\bigcup_{i=0}^{n}\mathcal{U}_{i}$ covers $X$ and each $\mathcal{U}_{i}$
is $r$-disjoint for $i=0,1,\cdots,n$. In this case, we say that
the asymptotic dimension of $X$ less than or equal to $n$, which is denoted by
 asdim$(X)\leq n$.

We say that asdim$(X)= n$ if  asdim$(X)\leq n$ and asdim$(X)\leq n-1$ is not true.
\end{defi}
\begin{remark}
There are other equivalent definitions of asymptotic
dimension. We guide the reader to \cite{Bell2011} for others.
\end{remark}

\begin{exa}\cite{Bell2011}

(1) asdim$(\mathbb{R}^n) =n$ for every $n \in\mathbb{N}$.\

(2) asdim$((m\mathbb{\mathbb{Z}})^n) =n$ for every $n,m \in\mathbb{N}$.\
\end{exa}

To classifying the spaces with infinite asymptotic dimension,
T. Radul generalized asymptotic dimension of a metric space $X$ to transfinite asymptotic dimension which is denoted by trasdim$(X)$ (see \cite{Radul2010}).
\begin{defi}
Let $Fin\mathbb{N}$ be the collection of all finite, nonempty
subsets of $\mathbb{N}$ and let $ M \subseteq Fin\mathbb{N}$. For $\sigma\in \{\varnothing\}\bigcup Fin\mathbb{N}$, let
$$M^{\sigma} = \{\tau\in Fin\mathbb{N} ~|~ \tau \cup \sigma \in M \text{ and } \tau \cap \sigma = \varnothing\}.$$

Let $M^a$ abbreviate $M^{\{a\}}$ for $a \in \NN$. Define the ordinal number Ord$M$ inductively as follows:
\begin{eqnarray*}
\text{Ord}M = 0 &\Leftrightarrow& M = \varnothing,\\
\text{Ord}M \leq \alpha &\Leftrightarrow& \forall~ a\in \mathbb{N}, ~\text{Ord}M^a < \alpha,\\
\text{Ord}M = \alpha &\Leftrightarrow& \text{Ord}M \leq \alpha \text{ and } \text{Ord}M < \alpha \text{ is not true},\\
\text{Ord}M = \infty &\Leftrightarrow& \text{Ord}M \leq\alpha \text{ is not true for every ordinal number } \alpha.
\end{eqnarray*}

\end{defi}
\begin{remark}
Note that if $ M,N\subseteq Fin\mathbb{N}$ and $M\subseteq N$, then $\text{Ord}M\leq\text{Ord}N.$
\end{remark}
\begin{defi}\cite{Radul2010}
Given a metric space $X$, let
\[
\begin{split}
A(X) = \{\sigma \in Fin\mathbb{N}~ |~&\text{ there are no uniformly bounded families } \mathcal{U}_i  \text{ for } i \in \sigma
 \\& \text{ such that each } \mathcal{U}_i
\text{ is } i\text{-disjoint and }\bigcup_{i\in\sigma}\mathcal{U}_i \text{~covers~} X\}.
\end{split}\]

The \emph{transfinite asymptotic dimension} of $X$ is defined as trasdim$(X)$=Ord$A(X)$.
\end{defi}
\begin{remark}
Note that trasdim$(X)\leq n$ if and only if asdim$(X)\leq n$.
\end{remark}
\begin{lem}\rm(see \cite{yanzhu2018}, Proposition 2.1)
\label{lem:trasdim}
Let $X$ be a metric space and let $l\in \mathbb{N}\cup \{0\}$, then the following conditions are equivalent:
\begin{itemize}
\item[\rm(1)] trasdim$(X) \leq \omega+l$;

\item[\rm(2)] For every $k\in\NN$, there exists $m=m(k)\in\NN$ such that for every~$n\in\NN$, there are uniformly bounded families $\mathcal{U}_{-l},\mathcal{U}_{-l+1},\cdots,\mathcal{U}_{m}$ satisfying $\mathcal{U}_i$ is $k$-disjoint for $i\in\{-l,\cdots, 0\}$, $\mathcal{U}_j$ is $n$-disjoint for $j\in\{1,2,\cdots, m\}$ and
$\bigcup_{i=-l}^{m}\mathcal{U}_i$ covers $X$.
\end{itemize}
\end{lem}
In \cite{yanzhu2018}, we introduced another approach to classify the metric spaces with infinite asymptotic dimension, which is called complementary-finite asymptotic dimension (coasdim).

\begin{defi}\cite{yanzhu2018}
Every ordinal number $\gamma$ can be represented as $\gamma=\lambda(\gamma)+n(\gamma)$, where $\lambda(\gamma)$ is the limit ordinal or $0$ and $n(\gamma)\in \mathbb{N}\cup \{0\}$. Let $X$ be a metric space, we define \emph{complementary-finite asymptotic dimension} coasdim$(X)$ inductively as follows:
\begin{itemize}
\item\text{coasdim}$(X)=-1$ iff $X=\emptyset$,

\item \text{coasdim}$(X)\leq \gamma$ iff for every $r>0$ there exist $r$-disjoint uniformly bounded families $\mathcal{U}_0,\ldots,\mathcal{U}_{n(\gamma)}$ of subsets of $X$ such that \text{coasdim}$(X\setminus \bigcup(\bigcup_{i=0}^{n(\gamma)}\mathcal{U}_i))<\lambda(\gamma)$,

\item \text{coasdim}$(X)=\gamma$ iff \text{coasdim}$(X)\leq \gamma$ and for every $\beta<\gamma$,
\text{coasdim}$(X)\leq\beta$ is not true.

\item \text{coasdim}$(X)=\infty$ iff for every ordinal number $\gamma$, \text{coasdim}$(X)\leq \gamma$ is not true.
\end{itemize}
$X$ is said to have \emph{complementary-finite asymptotic dimension} if coasdim$(X)\leq \gamma$ for some ordinal number $\gamma$.
\end{defi}
\begin{remark}
Note that coasdim$(X)\leq n$ if and only if asdim$(X)\leq n$.
\end{remark}
\begin{lem}\rm(see~\cite{yanzhu2018}, Theorem 3.3)
\label{lem:coasdim1}
Let $X$ be a metric space with $X_{1},X_{2}\subseteq X$, the
 coasdim$(X_{1}\cup X_{2})\leq \text{max}\{\text{coasdim}(X_{1}),\text{coasdim}(X_{2})\}$.
\end{lem}
\begin{defi}
Let $X$ and $Y$ be metric spaces. $f : X \rightarrow Y$ is called \emph{coarse embedding} if
there are non decreasing, unbounded functions $p_{1} ,p_{2} : \RR^{+}\rightarrow \RR^{+}$ such that for
every pair $x_{1},x_{2}\in X$,
\[
p_{1} (d(x_{1},x_{2}))\leq d(f(x_{1}),f(x_{2})) \leq p_{2}(d(x_{1},x_{2})).
\]
If additionally there exists $R>0$ such that $Y \subseteq N_{R}(f[X])$, then $f$ is called \emph{coarse equivalence} and metric spaces $X$ and $Y$
are said to be coarse equivalent. Coarse equivalence is an equivalence relation.

\end{defi}
\begin{lem}\rm(see~\cite{satkiewicz},\cite{yanzhu2018})
\label{lem:coarseequivalent}
Let $X$ and $Y$ be metric spaces. If $X$ and $Y$ are coarse equivalent, then
\[
 \text{trasdim}(X) =\text{trasdim}(Y) \text{ and coasdim}(X)=\text{ coasdim}(Y).
 \]
\end{lem}
\end{section}
\begin{section}{A metric space with transfinite asymptotic dimension $\omega +k$}\

In this section, we will construct a metric space $X_{\omega+k}$ and prove trasdim$(X_{\omega+k})=\omega+k$ for any $k\in\mathbb{N}$ inspired by the technique in \cite{omega+1}. However, the technique used here is not as simple as in \cite{omega+1}, so we need to introduce the concepts below.\

\begin{defi}(\cite{Engelking})
Let $X$ be a metric space and let $A,B$ be a pair of disjoint subsets of $X$. We say that a subset $L\subset X$ is a \emph{partition}
of $X$ between $A$ and $B$, if there exist open sets $U,W\subset X$ satisfying the following conditions
$$A\subset U, B\subset W\text{ and }X=U\sqcup L\sqcup W.$$
\end{defi}

\begin{defi}
Let $X$ be a metric space and let $A,B$ be a pair of disjoint subsets of $X$. For any $\epsilon>0$, we say that a subset $L\subset X$ is an \emph{$\epsilon$-partition}
of $X$ between $A$ and $B$, if there exist open sets $U,W\subset X$ satisfying the following conditions
$$A\subset U, B\subset W, X=U\sqcup L\sqcup W, d(L,A)>\epsilon\text{ and }d(L,B)>\epsilon$$
Clearly, an $\epsilon$-partition $L$ of $X$ between $A$ and $B$ is a partition of $X$ between $A$ and $B$.
\end{defi}

\begin{lem}
\label{partition2}
Let $L_0\doteq [0,B]^n$ for some $B>0$ and let $F_i^{+}$, $F_i^{-}$ be the pairs of opposite faces of $L_0$ for $i=1,2,\cdots,n$ and let $0<\epsilon<\frac{1}{6}B$.
For $k=1,2,\cdots, n$, let
$\mathcal{U}_k$ be an $\epsilon$-disjoint and $\frac{1}{3}B$-bounded family of subsets of $L_{k-1}$. Then there exists an $\epsilon$-partition $L_k$ of $L_{k-1}$ between $F_k^+\cap L_{k-1}$ and $F_k^-\cap L_{k-1}$ such that $L_k\subset L_{k-1}\cap(\bigcup \mathcal{U}_k)^c$.
\end{lem}

\begin{proof}
For any fixed $k\in \{1,\ldots,n\}$, let $\mathcal{A}_k\doteq \{U\in\mathcal{U}_k~|~d(U,F_k^+)\leq2\epsilon\}$ and  $\mathcal{B}_k\doteq \{U\in\mathcal{U}_k~|~ d(U,F_k^+)>2\epsilon\}$.
Note that $\mathcal{A}_k\cup\mathcal{B}_k=\mathcal{U}_k$.
Let
\[A_k=\bigcup\{N_{\frac{\epsilon}{3}}(U):U\in\mathcal{A}_k\}\text{ and }B_k=\bigcup\{N_{\frac{\epsilon}{3}}(U):U\in\mathcal{B}_k\}.\]
Then $d(A_k,F_k^-)>B-\frac{1}{3}B-2\epsilon-\frac{\epsilon}{3}>\frac{4}{3}\epsilon$ and $d(B_k,F_k^+)>2\epsilon-\frac{\epsilon}{3}>\frac{4}{3}\epsilon$. It follows that
\[(A_k\cup N_{\frac{4}{3}\epsilon}(F_k^+))\cap (B_k\cup N_{\frac{4}{3}\epsilon}(F_k^-))=\emptyset.\]
Let $L_k\doteq L_{k-1}\setminus ((A_k\cup N_{\frac{4}{3}\epsilon}(F_k^+))\cup (B_k\cup N_{\frac{4}{3}\epsilon}(F_k^-)))$, then $L_k$ is an $\epsilon$-partition of $L_{k-1}$ between $F_k^+\cap L_{k-1}$ and $F_k^-\cap L_{k-1}$ such that $L_k\subset L_{k-1}\cap(\bigcup \mathcal{U}_k)^c$.
\end{proof}

To prove the main result, we will use a version of Lebesgue theorem.
\begin{lem}\rm(see \cite{Engelking}, Lemma 1.8.19)
\label{partition}
Let $F_i^{+}$, $F_i^{-}$, where $i\in\{1,\ldots,n\}$, be the pairs of opposite faces of $I^n\doteq [0,1]^n$. If $I^n=L_0'\supset L_1'\supset \ldots\supset L_n'$ is a decreasing sequence of closed sets such that $L_i'$ is a partition of $L_{i-1}'$ between $L_{i-1}'\cap F_i^{+}$ and $L_{i-1}'\cap F_i^{-}$ for $i\in\{1,2,\ldots,n\}$, then $L_{n}'\neq \emptyset$.
\end{lem}

Similar with \cite{omega+1}, we will use asymptotic union to construct examples.\

\begin{defi}
Let $\{Z_{i}\}_{i=1}^{\infty}$ be a sequence of subspaces of a metric space $(Z,d_{Z})$.
Let \[
X=\bigsqcup_{i=1}^{\infty}(0,\cdots,0,Z_{i},0,\cdots).
\]
For every $x,y\in X$, there exist unique $l,k\in\NN$, $x_{l}\in Z_{l}$ and $y_{k}\in Z_{k}$ such that
$x=(0,\cdots,0,x_{l},0,\cdots)$ and $y=(0,\cdots,0,y_{k},0,\cdots)$.
Assume that $l\leq k$, put $c=0$ if $l=k$ and $c=l+(l+1)+...+(k-1)$ if $l<k$. Define a metric on $X$ by
\[d(x,y)=d_{Z}(x_{l},y_{k})+c.\]
We say that $(X,d)$ is \emph{asymptotic union }of $\{Z_{i}\}_{i=1}^{\infty}$, which is denoted by $\text{as}\bigsqcup_{i=1}^{\infty}Z_{i}$. And we denote $as\bigsqcup_{i=n}^{\infty}Z_i$ as a subspace of $as\bigsqcup_{i=1}^{\infty}Z_i$.
\end{defi}
\

Now we begin to construct a metric space with its transfinite asymptotic dimension $\omega+k$ for some $k\in\mathbb{N}$.
For any $k,i\in\mathbb{N}$, let
$$X_{\omega+k}^{(i)}=\{(x_1,\ldots,x_i)\in \mathbb{R}^i||\{j|x_j\notin  2^i\mathbb{Z}\}|\leq k\}.$$
Note that $X_{\omega+k}^{(i)}\subset \mathbb{R}^i$ for each $i\in \mathbb{N}$.
Let \[X_{\omega+k}=\text{as}\bigsqcup_{i=1}^{\infty}X_{\omega+k}^{(i)}, \]where $X_{\omega+k}^{(i)}$ is a subspace of the metric space $(\bigoplus\mathbb{R},d_{\text{max}})$ for each $i\in \mathbb{N}$. Note that $X_{\omega+l}\subseteq X_{\omega+k}$ with $l\leq k$.

\begin{prop}
\label{prop1}
For any $k\in\mathbb{N}$, trasdim$(X_{\omega+k})\leq\omega+k-1$ is not true.
\end{prop}

\begin{proof}
Suppose that trasdim$(X_{\omega+k})\leq\omega+k-1$. By Lemma \ref{lem:trasdim}, for every $n \in \mathbb{N}$, there exists $m=m(n) \in \mathbb{N}$ such that  there exist $B$-bounded
families $\mathcal{U}_{-k+1},\mathcal{U}_{-k+2},\ldots,\mathcal{U}_{m-1},\mathcal{U}_{m}$ satisfying $\mathcal{U}_i$ is $n$-disjoint for $i =-k+1,\ldots, 0$, $\mathcal{U}_j$ is $2^{m+k+2}$-disjoint for $j = 1, 2,\ldots, m$ and $\bigcup^{m}_{i=-k+1} \mathcal{U}_i$ covers $X_{\omega+k}$ and hence covers $[0,6B]^{m+k}\cap X_{\omega+k}^{(m+k)}$. Without lose of generality, we can assume $B=B(n)>$\text{max}$\{n,2^{m+k+2}\}$.\

We assume that $p=\frac{6B}{2^{m+k}}\in \mathbb{N}$.
Taking a bijection $\psi:\{1,2,\cdots,p^{m+k}\} \to \{0,1,2,\cdots,p-1\}^{m+k}$, let
$$Q(t)=\prod_{j=1}^{m+k}[2^{m+k}\psi(t)_{j},2^{m+k}(\psi(t)_{j}+1)],\text{in which }\psi(t)_{j}\text{ is the $j$th coordinate of }\psi(t).$$
Let $\mathcal{Q}=\{Q(t)~|~t\in\{1,2,\cdots,p^{m+k}\}\}$, then $[0,6B]^{m+k}=\bigcup_{Q\in\mathcal{Q}}Q$. Note that $[0,6B]^{m+k}\cap X_{\omega+k}^{(m+k)}=\bigcup_{Q\in \mathcal{Q}}\partial_{k}Q$, where $\partial_{k}Q$ is the $k$-dimensional skeleton of $Q$.

Let $L_0=[0,6B]^{m+k}$. By Lemma \ref{partition2},
since $N_{2^{m+k}}(\mathcal{U}_1)$ is $2^{m+k}$-disjoint and $(2^{m+k+1}+B)$-bounded, there exists a $2^{m+k}$-partition $L_1$ of $[0,6B]^{m+k}$ such that $L_1\subset (\bigcup N_{2^{m+k}}(\mathcal{U}_1))^c\cap [0,6B]^{m+k}$ and $d(L_1,F_1^{+/-})>2^{m+k}$.\
Since $L_1$ is a $2^{m+k}$-partition of $[0,6B]^{m+k}$ between $F_1^+$ and $F_1^-$, then $[0,6B]^{m+k}=L_1\sqcup A_1 \sqcup B_1$ such that $A_1$, $B_1$ are open in $[0,6B]^{m+k}$ and $A_1$, $B_1$ contain two opposite facets $F_1^-$, $F_1^+$ respectively.\

Let $\mathcal{M}_1=\{Q\in \mathcal{Q}|Q\cap L_1\neq \emptyset\}$ and $M_1=\bigcup \mathcal{M}_1$. Since $L_1$ is a $2^{m+k}$-partition of $[0,6B]^{m+k}$ between $F_1^+$ and $F_1^-$, then  $M_1$ is a partition of $[0,6B]^{m+k}$ between $F_1^+$ and $F_1^-$, i.e., $[0,6B]^{m+k}=M_1\sqcup A'_1 \sqcup B'_1$ such that $A'_1$, $B'_1$ are open in $[0,6B]^{m+k}$ and $A'_1$, $B'_1$ contain two opposite facets $F_1^-$, $F_1^+$ respectively.
Let $L'_1=\partial_{m+k-1}M_1=\bigcup\{\partial_{m+k-1}Q|Q\in \mathcal{M}_1\}$, then $[0,6B]^{m+k}\setminus (L'_1\sqcup A'_1 \sqcup B'_1)$ is the union of some disjoint open $(m+k)$-dimensional cubes with length of edge $= 2^{m+k}$.
So $L'_1$ is a partition of $[0,6B]^{m+k}$  between $F_1^+$ and $F_1^-$ and $L'_1\subset (\bigcup \mathcal{U}_1)^c\cap [0,6B]^{m+k}$.

For $N_{2^{m+k}}(\mathcal{U}_2)$, there exists a $2^{m+k}$-partition $L_2$ of $L'_1$ such that $L_2\subset (\bigcup  N_{2^{m+k}}(\mathcal{U}_2))^c\cap [0,6B]^{m+k}$ and $d(L_2,F_2^{+/-})>2^{m+k}$.\
Since $L_2$ is a $2^{m+k}$-partition of $L'_1$ between $L'_1\cap F_2^+$ and $L'_1\cap F_2^-$, then $L'_1=L_2\sqcup A_2 \sqcup B_2$ such that $A_2$, $B_2$ are open in $L'_1$ and $A_2$, $B_2$ contain two opposite facets $L'_1\cap F_2^-$, $L'_1\cap F_2^+$ respectively and $d(L_2,F_2^{+/-})>2^{m+k}$.

Let $\mathcal{M}_2=\{Q\in \mathcal{M}_1~|~Q\cap L_2\neq \emptyset\}$ and $M_2=\bigcup \mathcal{M}_2$. Since $L_2$ is a $2^{m+k}$-partition of $L'_1$ between $L'_1\cap F_2^+$ and $L'_1\cap F_2^-$, then $M_2 \cap L_1'$ is  a partition of $L'_1$  between $L'_1\cap F_2^+$ and $L'_1\cap F_2^-$, i.e., $L'_1=(M_2\cap L_1')\sqcup A'_2 \sqcup B'_2$ such that $A'_2$, $B'_2$ are open in $L'_1$ and $A'_2$, $B'_2$ contain two opposite facets $L'_1\cap F_2^-$, $L'_1\cap F_2^+$ respectively.
Let $L'_2=\partial_{m+k-2}M_2\doteq\bigcup\{\partial_{m+k-2}Q|Q\in \mathcal{M}_2\}$, then $L'_1\setminus (L'_2\sqcup A'_2 \sqcup B'_2)$ is the union of some disjoint open $(m+k-1)$-dimensional cubes with length of edge = $2^{m+k}$.
So $L'_2$ is also a partition of $L'_1$ between $L'_1\cap F_2^+$ and $L'_1\cap F_2^-$ and $L'_2\subset (\bigcup (\mathcal{U}_1\cup \mathcal{U}_2))^c\cap [0,6B]^{m+k}$.

After $m$ steps above, we have $L'_m$ to be a partition of $L'_{m-1}$ and $L'_m\subset (\bigcup(\mathcal{U}_1\cup \ldots\cup\mathcal{U}_m))^c\cap [0,6B]^{m+k}$. Note that $L'_m\subset X_{\omega+k}^{(m+k)}$ and hence $L'_m\subset (\bigcup(\mathcal{U}_{-k+1}\cup \ldots\cup\mathcal{U}_0))\cap [0,6B]^{m+k}$.

For $j=1,2,\cdots,k$, there exists a partition $L'_{m+j}$ of $L'_{m+j-1}$ between $L'_{m+j-1}\cap F_{m+j}^+$ and $L'_{m+j-1}\cap F_{m+j}^-$ such that
$L'_{m+j}\subseteq L'_{m+j-1}\cap(\bigcup(\mathcal{U}_{-j+1}\cup \ldots\cup\mathcal{U}_m))^c$ by Lemma \ref{partition2}.
It follows that $L'_{m+k}\subseteq L'_{m+k-1}\cap(\bigcup(\mathcal{U}_{-k+1}\cup \ldots\cup\mathcal{U}_0))^c=\emptyset$,
which is a contradiction with Lemma \ref{partition}. So trasdim$(X_{\omega+k})\leq \omega+k-1$ is not true.
\end{proof}

\begin{lem}\rm\label{lemmaOrd}
For $M\subset Fin\mathbb{N}$, $k\in \mathbb{N}\cup \{0\}$ and an infinite ordinal number $\alpha$,  \text{ Ord}$M\leq \alpha+k$ if and only if \text{ Ord}$M^{\tau}<\alpha$ for every $\tau\in Fin\mathbb{N}$ with $|\tau|=k+1$.
\end{lem}
\begin{proof}
We will prove it by induction on $k$.
Firstly, for $k=0$, the statement is true by the definition of\text{ Ord}$M$. Now assume that the statement is true for $k=n\in\mathbb{N}\cup \{0\}$.
Then by definition,  \text{ Ord}$M\leq \alpha+n+1$ if and only if for every $a\in\NN$, \text{ Ord}$M^{a}\leq \alpha+n$, i.e.,
\[\text{ Ord}M^{\sigma}\leq \alpha+n~~\text{for every}~ \sigma\in Fin\mathbb{N}\text{ with }|\sigma|=1.
\]
By assumption, it is equivalent to
\[\text{ Ord}(M^{\sigma})^{\varsigma}<\alpha~~\text{for every}~ \varsigma\in Fin\mathbb{N}\text{ with }\varsigma\cap\sigma=\emptyset\text{ and }|\varsigma|=n+1 \text{~and for every}~\sigma\in Fin\mathbb{N}\text{ with }|\sigma|=1.
\]
Therefore,
$$\text{ Ord}M\leq \alpha+n+1 \text{~if and only if~} \text{ Ord}M^{\tau}< \alpha~\text{for every}~ \tau\in Fin\mathbb{N}\text{ with }|\tau|=n+2.
$$
i.e., the result is true for $k=n+1$.
\end{proof}

\begin{lem}\rm(see~\cite{yanzhu2018}, Theorem 3.2)
\label{lem:coasdim}
Let $X$ be a metric space,  coasdim$(X)\leq \omega+k$  implies trasdim$(X)\leq\omega+k$ for every $k\in\mathbb{N}$.
\end{lem}

This lemma can be generalized as follows.

\begin{prop}\rm
\label{proprelation}
Let $X$ be a metric space, if coasdim$(X)\leq \gamma$ for some ordinal number $\gamma$, then trasdim$(X)\leq\gamma$.

\end{prop}
\begin{proof}
We will prove it by induction on $\gamma$.
\begin{itemize}
\item By the Lemma \ref{lem:coasdim}, we obtain that coasdim$(X)\leq \gamma$ implies trasdim$(X)\leq\gamma$ for every $\gamma<2\omega$.
\item Assume that the statement is true for $\gamma<\beta$. Now $\gamma=\beta=\alpha+n$ for some limit ordinal $\alpha$ and $n\in\mathbb{N}$.
We will show that trasdim$(X)\leq\alpha+n$, i.e., Ord$A(X)\leq\alpha+n$. \par
By the Lemma \ref{lemmaOrd}, it suffices to show that
\text{ Ord}$A(X)^{\tau}<\alpha$ for every $\tau\in Fin\mathbb{N}$ with $|\tau|=n+1$.
For any $\tau\in Fin\mathbb{N}$ with $|\tau|=n+1$, let $r=\text{max}$ $\{x\in\tau\}$. Since coasdim$(X)\leq \alpha+n$,
there exist $r$-disjoint uniformly bounded families $\mathcal{U}_0,\ldots,\mathcal{U}_{n}$ of subsets of $X$ such that \text{coasdim}$(X\setminus \bigcup(\bigcup_{i=0}^{n}\mathcal{U}_i))<\alpha<\beta$.
By assumption, it implies \text{trasdim}$(X\setminus \bigcup(\bigcup_{i=0}^{n}\mathcal{U}_i))<\alpha$, i.e., Ord$A(X\setminus \bigcup(\bigcup_{i=0}^{n}\mathcal{U}_i)))<\alpha$.

Note that $A(X)^{\tau}\subseteq A(X\setminus \bigcup(\bigcup_{i=0}^{n}\mathcal{U}_i)))$. Indeed, for every $\sigma\in A(X)^{\tau}$, if $\sigma\notin  A(X\setminus \bigcup(\bigcup_{i=0}^{n}\mathcal{U}_i))$, then there are uniformly bounded families $\mathcal{V}_{j}$ for $j\in\sigma$
such that $\mathcal{V}_{j}$ is $j$-disjoint and  $\bigcup_{j\in\sigma}\mathcal{V}_{j}$ covers $X\setminus \bigcup(\bigcup_{i=0}^{n}\mathcal{U}_i)$. It
follows that $(\bigcup_{j\in\sigma}\mathcal{V}_{j})\cup (\bigcup_{i=0}^{n}\mathcal{U}_i)$ covers $X$. Thus $\tau\sqcup\sigma\notin A(X)$,
i.e., $\sigma \notin A(X)^{\tau}$ which is a contradiction. So we obtain that $\sigma\in A(X\setminus \bigcup(\bigcup_{i=0}^{n}\mathcal{U}_i)))$.

Therefore, Ord$A(X)^{\tau}\leq\text{Ord}A(X\setminus \bigcup(\bigcup_{i=0}^{n}\mathcal{U}_i))<\alpha$.
i.e., the result is true for $\gamma=\beta$.
\end{itemize}

\end{proof}

For every $n,i,k\in\mathbb{N}$, let
$$X_{\omega+k}^{(i,n)}=\{(x_1,\ldots,x_i)\in \mathbb{R}^i~|~~|\{j~|~x_j\notin 2^n\mathbb{Z}\}|\leq k\}.$$
Note that $X_{\omega+k}^{(i)}=X_{\omega+k}^{(i,i)}$.\

\begin{lem}
\label{omega+1}
For every $r\in\NN$ with $r\geq4$, there exist $ n=r\in\NN$ and $r$-disjoint uniformly bounded families $\mathcal{U}_0,\mathcal{U}_1$ such that $\mathcal{U}_0\cup \mathcal{U}_1$ covers $as\bigsqcup_{i=n}^{\infty}X_{\omega+1}^{(i,n)}$.
\end{lem}

\begin{proof}
For every $r\in\NN$ and $r\geq4$, choose $n=r\in\NN$. For every $i\geq n$, let
\[\mathcal{U}_0^{(i)}=\{(\prod_{t=1}^i(n_t2^n-r,n_t2^n+r))\cap  X_{\omega+1}^{(i,n)}~|~n_t\in\ZZ\},\]
\[\mathcal{U}_1^{(i)}=\{(\prod_{t=1}^{j-1}(n_t2^n-r,n_t2^n+r)
\times [n_j2^n+r,(n_j+1)2^n-r]\times \prod_{t=j+1}^i(n_t2^n-r,n_t2^n+r))~\cap~  X_{\omega+1}^{(i,n)}~|~n_t\in\ZZ, 1\leq j\leq i\}\]
It is easy to see that $\mathcal{U}_0^{(i)}$ and $\mathcal{U}_1^{(i)}$ are $r$-disjoint and $2^n$-bounded families. Now for every $x=(x_1,\ldots,x_i)\in X_{\omega+1}^{(i,n)}\setminus(\bigcup\mathcal{U}_0^{(i)})$, there exists unique $j\in\{1,2,\cdots,i\}$ such that
$x_j\in [n_j2^n+r,(n_j+1)2^n-r]$. It follows that $x\in\mathcal{U}_1^{(i)}$. Therefore,
$\mathcal{U}_0^{(i)}\cup \mathcal{U}_1^{(i)}$ covers $X_{\omega+1}^{(i,n)}$. Let $\mathcal{U}_0=\bigcup_{i\geq n}\mathcal{U}_0^{(i)}$ and $\mathcal{U}_1=\bigcup_{i\geq n}\mathcal{U}_1^{(i)}$. Since $d(X_{\omega+1}^{(i,n)},X_{\omega+1}^{(j,n)})\geq n= r$ for every $i,j\geq n$ and $i\neq j$, then $\mathcal{U}_0,\mathcal{U}_1$  are $r$-disjoint and $2^n$-bounded families such that $\mathcal{U}_0\cup \mathcal{U}_1$ covers $as\bigsqcup_{i=n}^{\infty}X_{\omega+1}^{(i,n)}$.
\end{proof}

\begin{remark}
\label{remark}
By Lemma \ref{omega+1}, for every $r\in\NN$ and $r>1$, there exist $3r$-disjoint uniformly bounded families $\mathcal{U}_0,\mathcal{U}_1$ such that $\mathcal{U}_0\cup \mathcal{U}_1$ covers $as\bigsqcup_{i=3r}^{\infty}X_{\omega+1}^{(i,3r)}$. Let $\mathcal{V}_0=\{N_{r}(U)|~U\in\mathcal{U}_0\},\mathcal{V}_1=\{N_{r}(U)|~U\in\mathcal{U}_1\}$.
    Then $\mathcal{V}_0,\mathcal{V}_1$ are $r$-disjoint uniformly bounded families and $\mathcal{V}_0\cup \mathcal{V}_1$ covers $as\bigsqcup_{i=3r}^{\infty}N_{r}(X_{\omega+1}^{(i,n)})$, where $N_{r}(X_{\omega+1}^{(i,n)})$ is $r$-neighborhood of $X_{\omega+1}^{(i,n)}$ in $\mathbb{R}^i$. By similar argument, we obtain the following Lemma.

\end{remark}
\begin{lem}
\label{neighborofomega+t}
 For every $r\in\NN$ and $r>1$, there exist $n=3^{k-1}r\in\NN$ and $r$-disjoint uniformly bounded families $\mathcal{U}_0,\mathcal{U}_1,\ldots,\mathcal{U}_k$ such that
 $\mathcal{U}_0\cup\mathcal{U}_1\cup\ldots\cup\mathcal{U}_k$ covers
 $as\bigsqcup_{i=n}^{\infty}X_{\omega+k}^{(i,n)}$.
 \end{lem}
\begin{proof}
We will prove it by induction on $k$.
By Lemma \ref{omega+1}, the result is true for $k=1$. Assume that the result is true for $k= m$, then for every $r\in\NN$ and $r>1$, there exist $n=3^{m}r\in\NN$ and $3r$-disjoint uniformly bounded families $\mathcal{V}_0,\mathcal{V}_1,\ldots,\mathcal{V}_m$ such that
 $\mathcal{V}_0\cup\mathcal{V}_1\cup\ldots\cup\mathcal{V}_m$ covers
 $as\bigsqcup_{i=n}^{\infty}X_{\omega+m}^{(i,n)}$.
Now for $k=m+1$, let
$$\mathcal{U}_0=\{N_{r}(V)|~V\in\mathcal{V}_0\},\cdots,\mathcal{U}_m=\{N_{r}(V)|~V\in\mathcal{V}_m\}.$$
Then $\mathcal{U}_0,\mathcal{U}_1,\cdots, \mathcal{U}_m$ are $r$-disjoint and uniformly bounded families such that $\mathcal{U}_0\cup\mathcal{U}_1\cup\ldots\cup\mathcal{U}_m$ covers $as\bigsqcup_{i=n}^{\infty}N_{r}(X_{\omega+m}^{(i,n)})$.

Let
\[
\begin{split}
\mathcal{U}_{m+1}^{(i)}=\{\{x_{t}\}_{t=1}^{j_{1}-1}\times[n_{j_{1}}2^{n}+r,(n_{j_{1}}+1)2^{n}-r]\times(x_{t})_{t=j_{1}+1}^{j_{2}-1}\times[n_{j_{2}}2^{n}+r,(n_{j_{2}}+1)2^{n}-r]\times
\{x_{t}\}_{t=j_{2}+1}^{j_3-1}\times\\ \cdots\times\{x_{t}\}_{t=j_{m}+1}^{j_{m+1}-1} \times [n_{j_{m+1}}2^{n}+r,(n_{j_{m+1}}+1)2^{n}-r] \times
\{x_{t}\}_{t=j_{m+1}+1}^{i}|~x_{t}\in2^{n}\ZZ, n_{j_{k}}\in\ZZ, 1\leq k\leq m+1 \}.
\end{split}
\]
It is easy to see that $\mathcal{U}_{m+1}^{(i)}$ is $r$-disjoint and $2^{n}$-bounded.

Note that for every $i\geq n$,
\[
X_{\omega+m+1}^{(i,n)}\setminus \bigcup \mathcal{U}_{m+1}^{(i)}\subset N_{r}(X_{\omega+m}^{(i,n)})
\]
Indeed, for any $x=\{x_{t}\}_{t=1}^{i}\in X_{\omega+m+1}^{(i,n)}\setminus\bigcup\mathcal{U}_{m+1}^{(i)}$,
$\{x_{t}\}_{t=1}^{i}\in X_{\omega+m+1}^{(i,n)}$ implies that there exists at most $m+1$ coordinates $x_t$ such that $x_t\notin  2^n\ZZ$ and $x\notin  \bigcup\mathcal{U}_{m+1}^{(i)}$ implies that, among all the $x_t$ with $x_t\notin  2^n\ZZ$, there exists at least one $x_{t_0}$ such that $d(x_{t_0},2^n\ZZ)<r$. It follows that $x\in N_{r}(X_{\omega+m}^{(i,n)})$.

Since
\[d(X_{\omega+m+1}^{(i,n)},X_{\omega+m+1}^{(j,n)})>r \text{~for every~} i, j\geq n \text{~and~} i\neq j,\]
then $\mathcal{U}_{m+1}\doteq\bigcup_{i\geq n}\mathcal{U}_{m+1}^{(i)}$ is an $r$-disjoint uniformly bounded family of subsets and
\[
as\bigsqcup_{i=n}^{\infty}X_{\omega+m+1}^{(i,n)}\subset(\bigcup\mathcal{U}_{m+1})\cup\bigcup_{i=n}^{\infty}N_{r}(X_{\omega+m}^{(i,n)}).
\]
Therefore, $\mathcal{U}_0\cup\mathcal{U}_1\cup\ldots\cup\mathcal{U}_{m+1}$ covers
 $as\bigsqcup_{i=n}^{\infty}X_{\omega+m+1}^{(i,n)}$. So the result is true for $k=m+1$.
\end{proof}

\begin{prop}
\label{prop2}
coasdim$(X_{\omega+k})\leq \omega+k$.
\end{prop}
\begin{proof}
For every $r>0$, by Lemma \ref{neighborofomega+t}, there exist $n=n(r)\in\NN$ and $r$-disjoint uniformly bounded families $\mathcal{U}_0,\mathcal{U}_1,\ldots,\mathcal{U}_k$ such that
 $\mathcal{U}_0\cup\mathcal{U}_1\cup\ldots\cup\mathcal{U}_k$ covers
 $as\bigsqcup_{i=n}^{\infty}X_{\omega+k}^{(i,n)}$.  Since $X_{\omega+k}^{(i)}=X_{\omega+k}^{(i,i)}\subset X_{\omega+k}^{(i,n)}$ for $i\geq n$,
$X_{\omega+k}\setminus \bigcup(\mathcal{U}_0\cup\ldots\mathcal{U}_k)\subseteq as\bigsqcup_{i=1}^{n-1}X_{\omega+k}^{(i)}$.
By Lemma \ref{lem:coasdim1},
\[\text{coasdim}(X_{\omega+k}\setminus \bigcup(\mathcal{U}_0\cup\ldots\mathcal{U}_k))\leq\text{coasdim}(as\bigsqcup_{i=1}^{n-1}X_{\omega+k}^{(i)})\leq\text{coasdim}(as\bigsqcup_{i=1}^{n-1}\RR^{i})\leq\text{coasdim}(\RR^{n-1})<\omega.\]
So, by definition, coasdim$(X_{\omega+k})\leq \omega+k$.\end{proof}

\begin{thm}
\label{thm:trasdim}
trasdim$(X_{\omega+k})=\omega+k$ for every $k\in\NN$.
\end{thm}
\begin{proof}
By Lemma \ref{lem:coasdim} and the Proposition \ref{prop2},  trasdim$(X_{\omega+k})\leq\omega+k$.
By Proposition \ref{prop1}, we obtain that trasdim$(X_{\omega+k})=\omega+k$.
\end{proof}

\begin{thm}
\label{thm:coasdim}
coasdim$(X_{\omega+k})=\omega+k$ for every $k\in\NN$.
\end{thm}
\begin{proof}
By the Proposition \ref{prop1} and the Lemma \ref{lem:coasdim}, we can obtain that coasdim$(X_{\omega+k})\leq\omega+k-1$ is not true.
Then coasdim$(X_{\omega+k})=\omega+k$ by the Proposition \ref{prop2}.
\end{proof}

\end{section}

\begin{section}{A metric space with complementary-finite asymptotic dimension and transfinite asymptotic dimension $2\omega$}\

In this section, we will construct a metric space $Y_{2\omega}$ by taking asymptotic union of all the metric spaces $Y_{\omega+k}$ which is coarsely equivalent to $X_{\omega+k}$ for any $k\in\mathbb{N}$.\

For every $k,i\in\mathbb{N}$, let
\[Y_{\omega+k}^{(i)}=\{(x_1,\ldots,x_i)\in (2^k\mathbb{Z})^i~|~|\{j~|~x_j\notin 2^i\mathbb{Z}\}|\leq k\},
Y_{\omega+k}=\text{as}\bigsqcup_{i=1}^{\infty}Y_{\omega+k}^{(i)} \text{ and }
Y_{2\omega}=\text{as}\bigsqcup_{k=1}^{\infty}Y_{\omega+k},
\]
where $Y_{\omega+k}$ be a subspace of the metric space $\text{as}\bigsqcup_{j=1}^{\infty}\RR^{j}$ for each $j\in\mathbb{N}$.

\begin{prop}\rm
\label{prop3}
coasdim$(Y_{2\omega})=2\omega$ and trasdim$(Y_{2\omega})=2\omega$.
\end{prop}
\begin{proof}
For any $k\in\mathbb{N}$, since $Y_{\omega+k}\subseteq X_{\omega+k}$ and  $X_{\omega+k}\subseteq N_{2^{k}}(Y_{\omega+k})$, $Y_{\omega+k}$ and $X_{\omega+k}$ are coarse equivalent. Then
\[\text{
trasdim$(Y_{\omega+k})=$trasdim$(X_{\omega+k})=\omega+k$ and coasdim$(Y_{\omega+k})= \text{coasdim} (X_{\omega+k})=\omega+k$
}\]
by Lemma \ref{lem:coarseequivalent}.
It follows that coasdim$(Y_{2\omega})\geq2\omega$ and trasdim$(Y_{2\omega})\geq2\omega$.\

For every $n>0$, let
\[
\mathcal{U}=\{\{x\}~|~x\in as\bigsqcup_{k=n+1}^{\infty}Y_{\omega+k}\},
\]
then $\mathcal{U}$ is $n$-disjoint and uniformly bounded and
$$Y_{2\omega}\setminus \bigcup \mathcal{U}=as\bigsqcup_{k=1}^{n}Y_{\omega+k}.$$
Note that by Lemma~\ref{lem:coasdim1},
\[
\text{coasdim}(as\bigsqcup_{k=1}^{n}Y_{\omega+k})\leq\text{coasdim}(as\bigsqcup_{k=1}^{n}X_{\omega+k})\leq\text{max}_{1\leq k\leq n}\{\text{coasdim}(X_{\omega+k})\}=\omega+n<2\omega,
\]
so coasdim$(Y_{2\omega}\setminus \bigcup \mathcal{U})<2\omega.$ Then by the definition of complementary-finite asymptotic dimension,
coasdim$(Y_{2\omega})\leq2\omega$ and hence trasdim$(Y_{2\omega})\leq2\omega$ by Proposition \ref{proprelation}.
Therefore, coasdim$(Y_{2\omega})=2\omega$ and trasdim$(Y_{2\omega})=2\omega$.

\end{proof}
\begin{remark}
 Taras Radul has also constructed a metric space with the same property simultaneously in \cite{2omega}.
\end{remark}

\end{section}\

\begin{section}{Asymptotic dimension growth, transfinite asymptotic dimension and finite decomposition complexity}\

In this section, we will give examples to show that asymptotic dimension growth is independent of transfinite asymptotic and a similar result holds for asymptotic dimension growth and finite decomposition complexity.\

\begin{defi}\cite{polynomialdimensiongrowth}
Let $X$ be a metric space and let $\mathcal{U}$ be a cover of $X$.  The \emph{Lebesgue number} $L(\mathcal{U})$ of a cover $\mathcal{U}$ is defined to be
\[
\text{inf}_{x\in X}\text{sup}\{r>0~|~\exists~U\in\mathcal{U} \text{~such that~} B(x,r)\subseteq U\}.
\]
The \emph{multiplicity} $m(\mathcal{U})$ of a cover $\mathcal{U}$ is the maximal number of
elements of $\mathcal{U}$ with a nonempty intersection. We define a function\[
ad_{X}(\lambda) = \text{min}\{m(\mathcal{U})~|~\mathcal{U}\text{~ is a uniformly bounded cover of~} X \text{~and~}L(\mathcal{U})\geq\lambda \}- 1.\]
Clearly, $ad_{X}(\lambda)$ is monotone. We call this function the \emph{asymptotic dimension function}
of $X$.
\end{defi}
\

\begin{remark}
Let $X$ be a metric space. If $X$ has finite asymptotic dimension, then
$\text{lim}_{\lambda\rightarrow\infty}ad_{X}(\lambda)=$asdim$(X)$.
\end{remark}
\

Now we give an estimation of dimension growth of $Y_{2\omega}$.

\begin{prop}
\label{estr-dim}
For every $r>4$, $ad_{Y_{2\omega}}(\frac{r}{2})\leq \frac{r(r+3)}{2}+3^{r-1}r+1$.

\end{prop}
\begin{proof}
For every $r\in\NN$ and $r\geq4$, by Lemma \ref{neighborofomega+t}, there exist $n=3^{k-1}r\in\NN$ and $r$-disjoint uniformly bounded families $\mathcal{U}_0,\mathcal{U}_1,\ldots,\mathcal{U}_k$ such that
 $\mathcal{U}_0\cup\mathcal{U}_1\cup\ldots\cup\mathcal{U}_k$ covers
 $as\bigsqcup_{i=n}^{\infty}X_{\omega+k}^{(i,n)}$.
Since for $i\geq n$, $Y_{\omega+k}^{(i)}\subset X_{\omega+k}^{(i)}\subset X_{\omega+k}^{(i,n)}$, there are
$r$-disjoint uniformly bounded families  $\mathcal{U}_0^{(k)},...,\mathcal{U}_k^{(k)}$ such that
$\bigcup_{i=0}^k\mathcal{U}_i^{(k)}$ covers $as\bigsqcup_{i=n}^{\infty}Y_{\omega+k}^{(i)}$.
It follows that there are $\frac{r(r+3)}{2}$ uniformly bounded $r$-disjoint families $\{\mathcal{V}_{i}\}_{i=1}^{\frac{r(r+3)}{2}}$ whose union covers $as\bigsqcup_{1\leq k\leq r}(as\bigsqcup_{i=n}^{\infty}Y_{\omega+k}^{(i)})$.

Since $Y_{\omega+k}^{(i)}\subseteq \mathbb{R}^i$, asdim$(as\bigsqcup_{i=1}^nY_{\omega+k}^{(i)})\leq n$. By Lemma \ref{lem:coasdim1}, it follows that
\[
\text{asdim}(as\bigsqcup_{1\leq k\leq r}(as\bigsqcup_{i=1}^nY_{\omega+k}^{(i)}))\leq n= 3^{r-1}r.
\]
So there exist $r$-disjoint uniformly bounded families $\{\mathcal{V}_{i}\}_{i=\frac{r(r+3)}{2}+1}^{\frac{r(r+3)}{2}+3^{r-1}r+1}$ such that
$$\bigcup_{\frac{r(r+3)}{2}+1\leq i\leq\frac{r(r+3)}{2}+3^{r-1}r+1}\mathcal{V}_{i}=as\bigsqcup_{1\leq k\leq r}(as\bigsqcup_{i=1}^nY_{\omega+k}^{(i)}).$$
Let
$$\mathcal{V}_{0}=\{\{x\}|x\in as\bigsqcup_{ k> r}(as\bigsqcup_{i=1}^{\infty}Y_{\omega+k}^{(i)})\},$$
then $\mathcal{V}_{0}$ is an $r$-disjoint uniformly bounded family which covers $as\bigsqcup_{ k> r}(as\bigsqcup_{i=1}^{\infty}Y_{\omega+k}^{(i)})$.\

Let
$$\mathcal{U}=\bigcup_{0\leq i\leq \frac{r(r+3)}{2}+3^{r-1}r+1}N_{\frac{r}{2}}(\mathcal{V}_i).$$
Then $\mathcal{U}$ is a cover of $Y_{2\omega}$ with Lebesgue number $L(\mathcal{U})\geq\frac{r}{2}$ and multiplicity $m(\mathcal{U})\leq \frac{r(r+3)}{2}+3^{r-1}r+2$ and hence
$ad_{Y_{2\omega}}(\frac{r}{2})\leq \frac{r(r+3)}{2}+3^{r-1}r+1$.
\end{proof}
\begin{lem}\rm(\cite{Engelking} Theorem 1.8.20)
\label{lem:Lebesgue}
If $\mathcal{U}$ is a finite closed cover of the $n$-cube $I^{n}$ no member of which meets opposite faces of $I^{n}$,
then $\mathcal{U}$ contains $n+1$ members with a non-empty intersection.
\end{lem}

\begin{lem}\rm
\label{lem:greater}
 $ad_{(2^{k}\ZZ)^{n}}(2^{k+1})\geq n$ for every $k,n\in \mathbb{N}$.
\end{lem}
\begin{proof}
Suppose that $ad_{(2^{k}\ZZ)^{n}}(2^{k+1})\leq n-1$. Then there exist $B>0$ and $B$-bounded cover $\mathcal{U}$ of $(2^k\mathbb{Z})^n$ such that $L(\mathcal{U})\geq2^{k+1}$ and $m(\mathcal{U})\leq n$.

Let $N_{-2^k}(U)=U\setminus\overline{ N_{2^k}(X\setminus U)}=\{x\in U~|~d(x,X\setminus U)>2^k\}$ for each $U\in\mathcal{U}$.
Let $\mathcal{V}_{1}=\{N_{-2^k}(U)~|~U\in\mathcal{U}\}$. Since $L(\mathcal{U})\geq2^{k+1}$ and $m(\mathcal{U})\leq n$,  $\mathcal{V}_{1}$ is a $B$-bounded cover of $(2^k\mathbb{Z})^n$ such that $m(\mathcal{V}_{1})\leq n$.

Let
\[
\mathcal{V}=\{V\in\mathcal{V}_{1}~|~V\cap[0,B+2^{k+1}]^n\cap(2^k\mathbb{Z})^n\neq\emptyset\}.
\]
Since $[0,B+2^{k+1}]^n\cap(2^k\mathbb{Z})^n$ is a finite set and $m(\mathcal{V}_{1})\leq n$, $\mathcal{V}$ is a finite $B$-bounded cover of $[0,B+2^{k+1}]^n\cap(2^k\mathbb{Z})^n$ such that $m(\mathcal{V})\leq n$

Let
\[
\mathcal{W}=\{\overline{N_{2^{k-1}}(V)}~|~V\in\mathcal{V}\}.
\]
Then $\mathcal{W}$ is a finite closed and $(B+2^{k})$-bounded cover of $[0,B+2^{k+1}]^n$. Since each $W\in \mathcal{W}$ can not meet two opposite faces of $[0,B+2^{k+1}]^n$, then by Lemma \ref{lem:Lebesgue}, there exists $x\in [0,B+2^{k+1}]^n$ such that $x\in \bigcap_{i=1}^{n+1}\overline{N_{2^{k-1}}(V_{i})}$ for some $\overline{N_{2^{k-1}}(V_1)},\ldots,\overline{N_{2^{k-1}}(V_{n+1})} \in \mathcal{W}$. By definition, $V_i=N_{-2^k}(U_i)$ for some $U_i\in \mathcal{U}$, so  $d(x,N_{-2^k}(U_i))=d(x,V_i)\leq 2^{k-1}$ and hence $x\in U_i$, for $i\in\{1,\ldots,n+1\}$.
Therefore, $x\in \bigcap_{i=1}^{n+1}U_i$ which is a contradiction with $m(\mathcal{U})\leq n$.

\end{proof}
Since $ad_{X}(\lambda)$ is monotone and $ad_{X}(\lambda)\leq $asdim$(X)$ for every $\lambda>0$, we obtain the following result.

\begin{coro}
\label{coro1}
Let $k,n\in\NN$, $ad_{(2^{k}\ZZ)^n}(r)=n$ for every $r\geq2^{k+1}$.
\end{coro}

Let $g:\RR^{+}\rightarrow\ZZ^{+}$ be an increasing function satisfying $\lim_{r\to +\infty}g(r)=+\infty$.
By Corollary \ref{coro1}, $ad_{(2^k\mathbb{Z})^{g(r)}}(r)=g(r)$ for every $k\leq\text{log}_{2}r-1$.
Let
$$\tilde{g}(r)=\max\{k\in\mathbb{N}~|~ad_{(2^k\mathbb{Z})^{g(r)}}(r)=g(r)\}.$$
Then $\tilde{g}(r)\geq[\text{log}_{2}r-1]>\text{log}_{2}r-2$.
Therefore $\tilde{g}(r)$ is an positive integer value  function with $\lim_{r\to +\infty}\tilde{g}(r)=+\infty$.
Define \[X_{\omega}(g)=as\bigsqcup_{i=1}^{\infty}(2^{\tilde{g}(i)}\mathbb{Z})^{g(i)}.\]

\begin{prop}
\label{omegawithanyfunctiongrowth}
For any increasing function $g:\RR^{+}\rightarrow\ZZ^{+}$ with $\lim_{r\to +\infty}g(r)=+\infty$, trasdim$(X_{\omega}(g))=\omega$ and $ad_{X_{\omega}(g)}(r)\geq g(r)$ for every $r>0$.
\end{prop}
\begin{proof}
For every $r>0$, since $\lim_{i\to +\infty}\widetilde{g}(i)=+\infty$, there exists $M\in\NN$ such that $\tilde{g}(i)>r$ when $i\geq M$. Then
$$\mathcal{U}=\{\{x\}|x\in as\bigsqcup_{i=M}^{\infty}(2^{\tilde{g}(i)}\mathbb{Z})^{g(i)}\},$$
is an $r$-disjoint and uniformly bounded family of subsets. Since
\[\text{
asdim$(X_{\omega}(g)\setminus(\bigcup\mathcal{U}))=\text{asdim}(as\bigsqcup_{i=1}^{M-1}(2^{\tilde{g}(i)}\mathbb{Z})^{g(i)})<\infty$ and asdim$(X_{\omega}(g))=\infty$,
}\]
coasdim$(X_{\omega}(g))=\omega$. By Proposition \ref{proprelation}, trasdim$(X_{\omega}(g))=\omega$ . While by the definition of $\tilde{g}$,
\[
ad_{X_{\omega}(g)}(r)\geq ad_{(2^{\tilde{g}(r)}\mathbb{Z})^{g(r)}}(r)=g(r).
\]
\end{proof}
\begin{remark}
Let $g(r) =ad_{Y_{2\omega}}(r)+1$, then
\[\text{
$ad_{X_{\omega}(g)}(r)\geq g(r) > ad_{Y_{2\omega}}(r)$.
}\] But trasdim$(Y_{2\omega})=2$trasdim$(X_{\omega}(g))$.
\end{remark}
\

In fact, we can construct a metric space $X$ with any asymptotic dimension growth with trasdim$(X)=\omega$. A similar result holds for asymptotic dimension growth and finite decomposition complexity which is introduced in \cite{Yu2013}.\

\begin{defi}
(\cite{Yu2013})\\
(1) Let $\mathcal{D}_0$ be the collection of bounded families: $\mathcal{D}_0 = \{\mathcal{X} : \mathcal{X} \text{ is bounded}\}$.\\
(2) Let $\alpha$ be an ordinal greater than 0, let $\mathcal{D}_{\alpha}$ be the collection of metric families decomposable
over $\bigcup_{\beta<\alpha}\mathcal{D}_{\beta}:$
$$ \mathcal{D}_{\alpha} = \{\mathcal{X} : \forall~ r > 0, \exists~ \beta<\alpha, \exists ~\mathcal{Y}\in \mathcal{D}_{\beta}, \text{ such that } \mathcal{X}\stackrel{r}{\rightarrow} \mathcal{Y}\}.$$
\end{defi}

\begin{prop}
\label{omegawithanyfunctiongrowth2}
For any positive integer value increasing function $g$ with $\lim_{r\to +\infty}g(r)=+\infty$, $X_{\omega}(g)\in\mathcal{D}_{\omega}$ and $ad_{X_{\omega}(g)}(r)\geq g(r)$ for every $r>0$.
\end{prop}
\begin{proof}
We only need to prove $X_{\omega}(g)\in\mathcal{D}_{\omega}$. For every $r>0$, since $\lim_{i\to +\infty}\widetilde{g}(i)=+\infty$, there exists $M\in\NN$ such that $\tilde{g}(i)>r$ when $i\geq M$. Then
$$\mathcal{U}=\{\{x\}|x\in as\bigsqcup_{i=M}^{\infty}(2^{\tilde{g}(i)}\mathbb{Z})^{g(i)}\},$$
is an $r$-disjoint and uniformly bounded family of subsets.
Let $\mathcal{V}=\mathcal{U}\cup \{as\bigsqcup_{i=1}^{M-1}X_{\omega}^{(i)}(g)\}$.
Then $X_{\omega}(g)\stackrel{r}{\longrightarrow}\mathcal{V}$.
Since $as\bigsqcup_{i=1}^{M-1}X_{\omega}^{(i)}(g)$ has finite asymptotic dimension, $\mathcal{V}\in \mathcal{D}_{n}$ for some $n\in\NN$.
Hence $X_{\omega}(g)\in \mathcal{D}_{\omega}.$

\end{proof}

\end{section}

\providecommand{\bysame}{\leavevmode\hbox to3em{\hrulefill}\thinspace}
\providecommand{\MR}{\relax\ifhmode\unskip\space\fi MR }
\providecommand{\MRhref}[2]{%
  \href{http://www.ams.org/mathscinet-getitem?mr=#1}{#2}
}
\providecommand{\href}[2]{#2}

\end{document}